\documentclass[a4paper,11pt]{amsart}
\usepackage[top=1.25in, bottom=1.25in, left=1.25in, right=1.25in]{geometry}
\usepackage{graphicx}
\usepackage{amsfonts}
\usepackage{epsf}
\usepackage{amssymb}
\usepackage{amsmath}
\usepackage{amscd}
\usepackage{amsthm}
\usepackage{tikz}
\usetikzlibrary{
  cd,
  calc,
  positioning,
  arrows,
  decorations.pathreplacing,
  decorations.markings,
}
\usepackage{pdfpages}
\usepackage{fancyhdr}
\usepackage{setspace}
\usepackage{hyperref}
\usepackage[all]{xy}
\usetikzlibrary{matrix}
\usepackage{verbatim}
\usepackage{enumerate}
\usepackage{mathrsfs}
\usepackage{pinlabel}
\usepackage{lscape}
\usepackage{color}
\usepackage{subfigure}
\usepackage[normalem]{ulem}

\newtheorem{theorem}{Theorem}[section]

\newtheorem{proposition}[theorem]{Proposition}
\newtheorem{corollary}[theorem]{Corollary}
\newtheorem{lemma}[theorem]{Lemma}

\newtheorem{question}[theorem]{Question}

\newtheorem{innercustomgeneric}{\customgenericname}
\providecommand{\customgenericname}{}
\newcommand{\newcustomprop}[2]{%
  \newenvironment{#1}[1]
  {%
   \renewcommand\customgenericname{#2}%
   \renewcommand\theinnercustomgeneric{##1}%
   \innercustomgeneric
  }
  {\endinnercustomgeneric}
}
\newcustomprop{customprop}{Proposition}

\theoremstyle{definition}
\newtheorem{definition}[theorem]{Definition}
\newtheorem{example}[theorem]{Example}

\newtheorem*{case2'}{Case 2$'$}
\newtheorem{theorem-named}{}
\newtheorem{theorem-labeled}{Theorem}
\newtheorem{definition-named}{}
\newtheorem{conjecture-named}{}
\newtheorem{case-named}{}

\numberwithin{equation}{section}

\newcommand{\tors}{{\rm Tors}}

\newcommand{\Z}{\mathbb{Z}}

\newcommand{\F}{\mathbb{F}}

\newcommand{\s}{\mathfrak{s}}

\def\d{\partial}
\def\Z{\mathbb{Z}}

\def\Im{\operatorname{Im}}

\makeatletter

\def\s{\mathcal{F}_{sh}}
\def\f{\mathcal{F}}
\def\h{\mathcal{F}_{h}}

\def\CFKi{\mathit{CFK}^\infty}
\def\Ord{\operatorname{Ord}}
\def\HFKh{\widehat{\mathit{HFK}}}
\def\HFKm{\mathit{HFK}^-}
\def\ic{\widehat{\mathit{HF}}}
\def\CFKF{\mathit{CFK}_{\F[U,V]}}

\begin{document}
\title{Ribbon knots, cabling, and handle decompositions}

\author{Jennifer Hom}
\thanks{The first author was partially supported by NSF grant DMS-1552285.}
\address{School of Mathematics, Georgia Institute of Technology, Atlanta, GA, USA}
\email{hom@math.gatech.edu}

\author{Sungkyung Kang}
\address{Institute of Mathematical Sciences, The Chinese University of Hong Kong, Shatin, N.T. Hong Kong}
\email{skkang@math.cuhk.edu.hk}

\author{JungHwan Park}
\address{School of Mathematics, Georgia Institute of Technology, Atlanta, GA, USA}
\email{junghwan.park@math.gatech.edu }

\def\subjclassname{\textup{2020} Mathematics Subject Classification}
\expandafter\let\csname subjclassname@1991\endcsname=\subjclassname
\expandafter\let\csname subjclassname@2000\endcsname=\subjclassname

\begin{abstract}
The fusion number of a ribbon knot is the minimal number of 1-handles needed to construct a ribbon disk. The strong homotopy fusion number of a ribbon knot is the minimal number of 2-handles in a handle decomposition of a ribbon disk complement. We demonstrate that these invariants behave completely differently under cabling by showing that the $(p,1)$-cable of \emph{any} ribbon knot with fusion number one has strong homotopy fusion number one and fusion number~$p$. Our main tools are Juh{\'a}sz-Miller-Zemke's bound on fusion number coming from the torsion order of knot Floer homology and Hanselman-Watson's cabling formula for immersed curves.
\end{abstract}

\maketitle

\section{Introduction}\label{sec:intro}
A knot is \emph{slice} if it bounds a smoothly embedded disk in $B^4$. Moreover, if a slice knot $K$ bounds a smoothly embedded disk $D$ in $B^4$ for which there are no local maxima of the height function restricted to the disk, then we say that $K$ is \emph{ribbon} and $D$ is a \emph{ribbon disk}. It is an outstanding open
problem due to Fox \cite{Fox:1961-1} whether the two notions coincide. 

One can also define notions that lie in between the two. A knot is said to be \emph{strongly homotopy ribbon} if it bounds  a smoothly embedded disk in $B^4$ such that the disk complement has a handlebody decomposition consisting of only 0, 1 and 2-handles (see e.g.\ \cite{Cochran:1983-1, Larson-Meier:2015-1,Miller-Zemke:2019-1}). Also, a knot is said to be \emph{homotopy ribbon} if it bounds a smoothly embedded disk in $B^4$ such that the fundamental group of the knot complement surjects to the fundamental group of the disk complement \cite{Casson-Gordon:1983-1}.\footnote{Note that the original definition of \cite{Casson-Gordon:1983-1} assumed that the 4-manifold was a homotopy 4-ball, but for our purposes we will assume that the 4-manifold is $B^4$.} Hence we have the following chain of implications:
$$
K \text{ is ribbon} \Rightarrow K \text{ is strongly homotopy ribbon} \Rightarrow K \text{ is homotopy ribbon}\Rightarrow K \text{ is slice}.$$ Whether the converse of any of these implications holds is an interesting open problem. 

In this article, we study the complexity of ribbon disks using concepts analogous to those above. Thus, we introduce the following terminologies. First, recall that for any ribbon knot $K$, the \emph{fusion number} $\f(K)$ is defined to be the minimal number of 1-handles needed to construct a ribbon disk in $B^4$ (see e.g.\ \cite{Miyazaki:1986-1}). We define the \emph{strong homotopy fusion number} $\s(K)$ to be the minimal number of 2-handles in a handle decomposition of a ribbon disk complement and the \emph{homotopy fusion number} $\h(K)$ to be the minimal number of relations for a presentation of the fundamental group of a ribbon disk complement. 

Given a ribbon disk $D$ in $B^4$, let $B^4 \smallsetminus D$ denote the complement of an open tubular neighborhood of $D$ in $B^4$. Recall that if a ribbon disk $D$ in $B^4$ has $n$ 1-handles, then $B^4 \smallsetminus D$ has a handle decomposition with $n$ 2-handles. Further, if $B^4 \smallsetminus D$ has a handle decomposition with $n$ 2-handles, then $\pi_1(B^4 \smallsetminus D)$ has a presentation with $n$ relations. Hence for each ribbon knot $K$, we have the following inequalities 
$$\h(K) \leq  \s(K) \leq  \f(K).$$

Seemingly, the fusion number and the strong homotopy fusion number are intimately related. Nevertheless, in this article, we show that these invariants behave completely differently under one of the most basic operations on knots, \emph{cabling}. Let $K_{p,q}$ denote the $(p, q)$-cable of $K$, where $p$ denotes the longitudinal winding, and let $K_{p_1, q_1; p_2, q_2; \dots ; p_n, q_n}$ denote the iterated cable of $K$. We assume throughout that $p>1$ and $p_i>1$. 

%

\begin{theorem}\label{thm:main}If $K$ is ribbon with $\f(K)=1$, then $\s(K_{p,1})=1$ and $\f(K_{p,1})=p$. Furthermore, $\s(K_{p_1,1; \dots ; p_n, 1})=1$ and $\f(K_{p_1,1; \dots ; p_n, 1})=p_1 p_2 \dots p_n$.
\end{theorem}

Note that there are many ribbon knots with fusion number one. In fact, it is known that any ribbon knot with fewer than 11 crossings has fusion number one~\cite[Appendix F]{Kawauchi:1996-1}. Further, if $J$ is a 2-bridge knot and $\overline J$ is the reverse of the mirror image of $J$, then $J \#\overline J$ has fusion number one. Also, the Kinoshita-Terasaka knot has fusion number one.

Theorem~\ref{thm:main} gives, in particular, examples of ribbon knots with strong homotopy fusion number one and arbitrarily large fusion number. We remark that it is also possible to produce such examples by combining previously known results. It was shown in \cite[Section 1.7]{JMZ} that $Q_{p,q}:=T_{p,q} \# \overline T_{p,q}$, where $\overline T_{p,q}$ is the reverse of the mirror image of $T_{p,q}$, has fusion number $\min\{p,q\} -1$. Moreover, if follows from Meier and Zupan \cite[Proposition 5.3]{Meier-Zupan:2019-1} that $\s(Q_{p,q})=1$.

The main tool we use to prove Theorem~\ref{thm:main} is called the torsion order (see Definition~\ref{def:order}) coming from knot Floer homology \cite{ozsvath2004holomorphicknot, Rasmussen:2003-1}. We denote the torsion order of $K$ by $\Ord_U(K)$. This invariant was defined by \cite{JMZ} and has many topological applications. For instance, it gives lower bounds on the bridge number and the band unlinking number of a knot. Also, it gives a lower bound on the fusion number of a ribbon knot and we use this property to establish Theorem~\ref{thm:main}. The computation of this invariant is based on bordered Floer homology \cite{LOT}, interpreted in terms of immersed curves as in \cite{HRW, HRW:2018}. In fact, we provide a lower bound on the torsion order of cable knots, which may be of independent interest. Let $\CFKF(K)$ denote the knot Floer complex over $\F[U,V]$.
We say that a knot $K$ has a \emph{unit box} if $\CFKF(K)$ contains a direct summand generated by four elements $\{a, b, c, d\}$, where the differential acts by $$\d a= Ub + V c, \d b = Vd, \text{ and } \d c = U d.$$

\begin{proposition}\label{prop:introcables}
If $K$ has a unit box, then
\[ \Ord_U(K_{p_1, q_1; p_2, q_2; \dots ; p_n, q_n}) \geq p_1  p_2 \dots  p_n. \]
\end{proposition}

It is well known that the Kinoshita-Terasaka knot $\mathcal{K}$ bounds a ribbon disk $D$ where $\pi_1(B^4\smallsetminus D) \cong \mathbb{Z}$. In particular, we have $\h(\mathcal{K})=0$ and $\s(\mathcal{K})=1$. Hence we ask the following natural question.

\begin{question}\label{q:shvsh} Are there knots for which the difference $\s-\h$ is arbitrarily large?
\end{question}

In Sections~\ref{sec:shfusion}, we show that $\mathcal{F}(K_{p,1})\le p\mathcal{F}(K)$ for any ribbon knot $K$ (see Proposition~\ref{prop:handle}). Theorem~\ref{thm:main} shows that the inequality is sharp for any ribbon knot with fusion number one, as well as for any iterated $(p_i, 1)$-cable of such a knot. One can ask if the inequality is sharp in general.

\begin{question}\label{q:shvsh} If $K$ is ribbon, then $\mathcal{F}(K_{p,1})=p\mathcal{F}(K)$?
\end{question}

Lastly, we make a remark on previously known results. There were several lower bounds on the fusion number prior to the torsion order. Most of them come from cyclic branched covers of knots \cite{Nakanishi-Nakagawa:1982-1, Kanenobu:2010-1, Aceto-Golla-Lecuona:2018-1}. For instance, given any knot $K$, the minimum number of generators for $H_1(\Sigma(K);\Z)$ is a lower bound on $2\f(K)$,  where $\Sigma(K)$ is the 2-fold cyclic branched cover of $K$ \cite[Proposition 2]{Nakanishi-Nakagawa:1982-1}. This, in particular, implies that if $K$ is a ribbon knot and $\Sigma(K)$ is not an integral homology sphere, then the connected sum of $n$ copies of $K$ has fusion number at least $n$. Moreover, it is straightforward to verify that these lower bounds coming from cyclic branched covers are, in fact, also lower bounds on $\s$. Note that the Kinoshita-Terasaka knot has trivial Alexander polynomial, and thus any cyclic branched cover of the $(p,1)$-cable of the Kinoshita-Terasaka knot is an integral homology sphere. In this case, none of the previous bounds apply.

\subsection*{Organization}
We work in the smooth category in this article. In Section~\ref{sec:shfusion}, we establish an upper bound on $\s$ using Kirby calculus. In Section~\ref{sec:cables}, we prove Proposition~\ref{prop:introcables} and  Theorem~\ref{thm:main}.

\subsection*{Acknowledgements}
This project began when the second author
was visiting Georgia Tech. The second author would like to thank Georgia Tech for its hospitality during his visit. The authors would also like to thank Christopher Davis and John Etnyre for helpful conversations. Lastly, we thank Jeffrey Meier and Alexander Zupan for directing us to their work.

\section{The strong homotopy fusion number of cables}\label{sec:shfusion}

In this section, we prove the following proposition.
\begin{proposition}\label{prop:handle}If $K$ is ribbon, then $\mathcal{F}_{sh}(K_{p,1})\le\mathcal{F}(K)$ and $\mathcal{F}(K_{p,1})\le p\mathcal{F}(K)$.
\end{proposition}

Before proving the proposition, we briefly explain how to obtain a Kirby diagram of a ribbon disk complement from a description of a ribbon disk. For more details, see, for example, \cite[Section 6.2]{GompfStipsicz}. Suppose that $K$ is a ribbon knot with a ribbon disk $D$ such that $D$ in $B^4$ has $n+1$ 0-handles, and $n$ 1-handles. Then there exists a movie representing $D$, which starts with $n+1$ births, followed by $n$ saddles. We can draw births as small unknotted components and saddles as red arcs, which gives us the diagram on the top of Figure~\ref{Fig1}. Here each red arc is framed. Now, we replace each unknotted component by a dotted circle and turn each saddle into a 0-framed simple closed curve, as drawn on the bottom of Figure~\ref{Fig1}. The resulting diagram is a Kirby diagram representing the disk complement $B^{4}\smallsetminus D$.

\begin{figure}[h]
\resizebox{.6\textwidth}{!}{\includegraphics{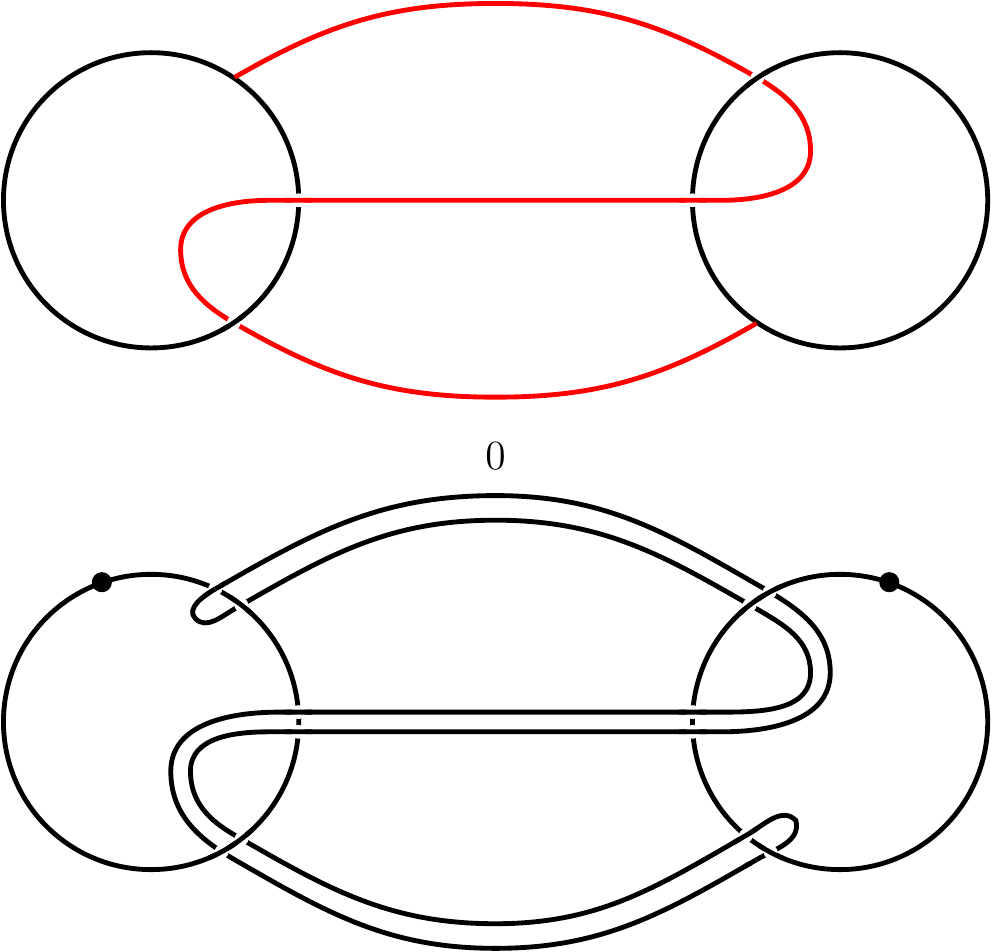}}
\caption{\label{Fig1} Top, a diagram representing births and saddles of a ribbon
disk $D$ for the stevedore knot (here, we are using the blackboard framing for the red arc). Bottom, the induced Kirby diagram representing $B^{4}\smallsetminus D$.}
\end{figure}

From the Kirby diagram of $B^{4}\smallsetminus D$ obtained from a movie of $D$, we can also see how the ribbon knot $K$ sits in $\partial B^{4}$. Instead of replacing all $n+1$ unknotted components with dotted circles, we do the replacement process except for one unknotted component, and label the remaining unknotted component as $K$. This resulting diagram is drawn on the top of Figure~\ref{Fig2}. It is straightforward to check that the diagram we get is indeed the Kirby diagram of $B^{4}$, and the labeled unknotted component is isotopic to the given ribbon knot $K$. Moreover, the obvious disk that $K$ bounds in the 0-handle of the handle decomposition of $B^4$, which is described by the Kirby diagram, is isotopic to the ribbon disk $D$. Note that the embedding of the disk $D$ in the 0-handle is trivial. By trivial, we mean that the disk can be isotoped to the boundary of the 0-handle. Now, we are ready to prove Proposition~\ref{prop:handle}.

\begin{figure}[h]
\resizebox{0.6\textwidth}{!}{\includegraphics{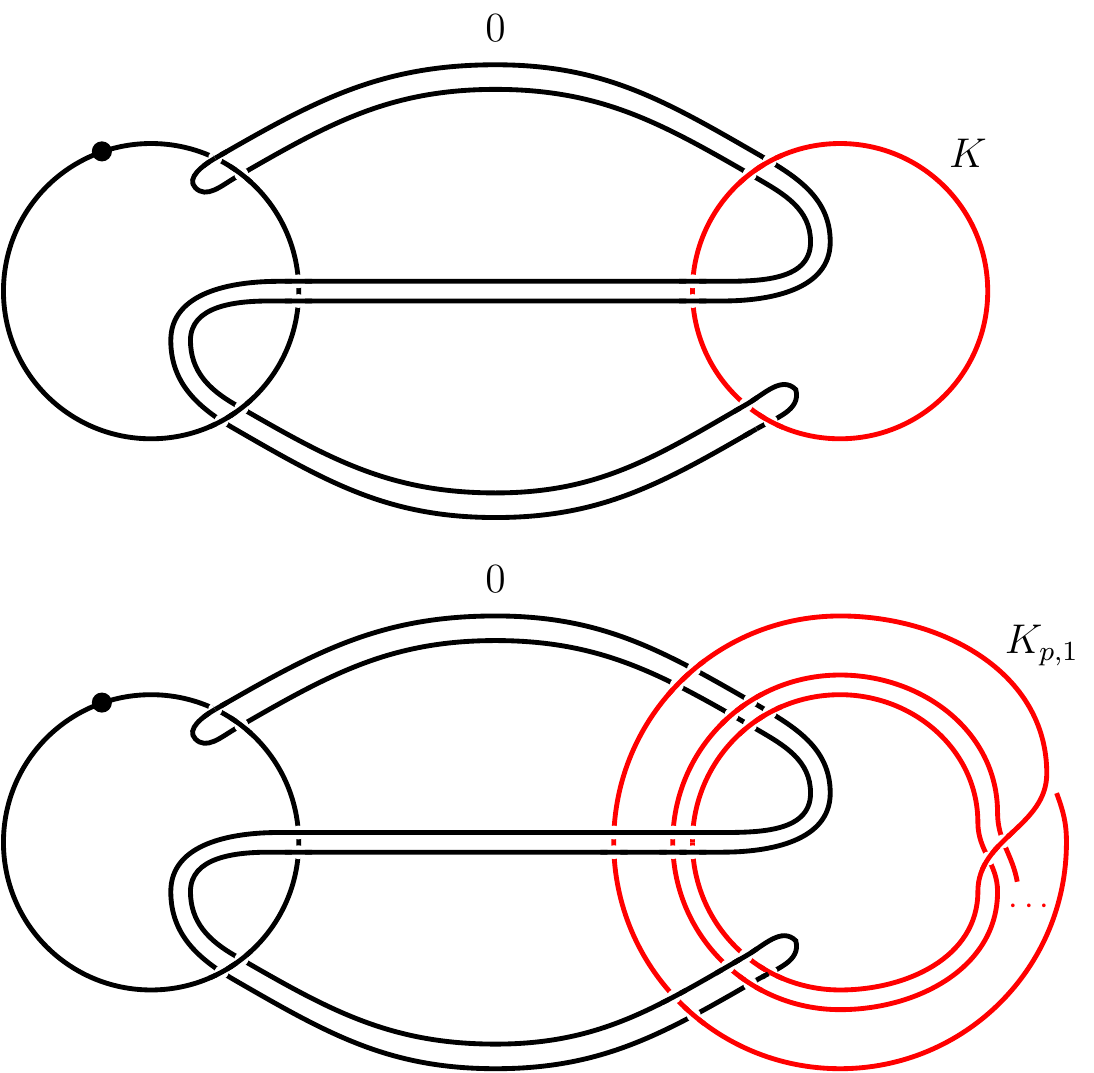}}
\caption{\label{Fig2}The knots $K$ and $K_{p,1}$ in the Kirby diagram
representing $B^{4}$.}
\end{figure}

\begin{proof}[Proof of Proposition \ref{prop:handle}] Let $K$ be a ribbon knot with a ribbon disk $D$. We may assume that $D$ in $B^4$ has $\f(K)+1$ 0-handles and $\f(K)$ 1-handles. Consider the Kirby diagram of $B^{4}$ together with the labeled unknotted component representing $K$ (the red curve on the top of Figure~\ref{Fig2}), as explained above. Replacing the unknotted component by its $(p,1)$-cable, as shown on the bottom of Figure~\ref{Fig2}, gives a diagram for $K_{p,1}$. Note that the curve replaced by its $(p,1)$-cable is still unknotted in the diagram. Let $D_p$ be the ribbon disk obtained by attaching $p-1$ half-twisted bands to $p$ parallel copies of $D$. Since $D$ is embedded trivially in the 0-handle, so is $D_p$. Hence we can replace $K_{p,1}$ by a dotted circle to get a new Kirby diagram for a handle decomposition of $B^4 \smallsetminus D_p$, whose number of 2-handles is again $\mathcal{F}(K)$. Therefore we conclude that $\mathcal{F}_{sh}(K_{p,1})\le\mathcal{F}(K)$.

\begin{figure}[h]
\resizebox{.6\textwidth}{!}{\includegraphics{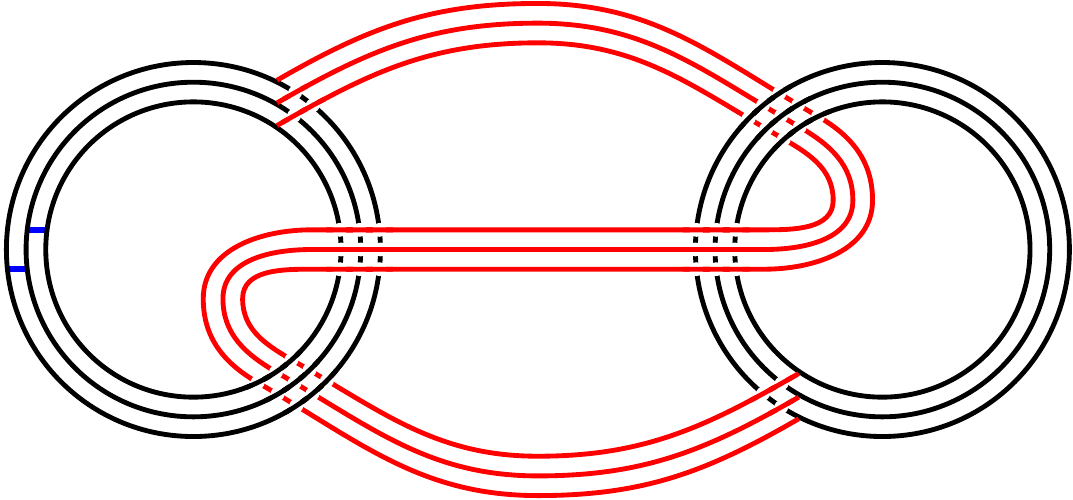}}
\caption{\label{fig:Dp1} A description of a ribbon disk $D_p$ obtained by adding $p-1$ 1-handles (blue arcs) to the $p$-parallel copies of $D$.}
\end{figure}

Now, we show the second inequality. First, note that the ribbon disk $D_p$ in $B^4$ has $p\f(K)+p$ 0-handles and $p\f(K)+p-1$ 1-handles. Here, $p\f(K)$ 1-handles are obtained from taking $p\f(K)$ parallel copies of 1-handles of $D$ in $B^4$ (see the red arcs of Figure~\ref{fig:Dp1}) and $p-1$ 1-handles corresponding to the cabling operation (see the blue arcs of Figure~\ref{fig:Dp1}). As the $(p,1)$-cable of the unknot is again the unknot, we see that $p$ 0-handles and $p-1$ 1-handles corresponding to the cabling operation simplify to a single 0-handle. This completes the proof.\end{proof}
\section{The fusion number of cables}\label{sec:cables}

First, we quickly recall some definitions and notations from knot Floer homology; see \cite{Manolescu:2016-1, Hom:2017-1} for survey articles on this subject. Throughout, we work over $\mathbb{F}=\mathbb{Z}/2\mathbb{Z}$ and use the convention that $\mathbb{N} = \mathbb{Z}_{\geq 0}$. 

We write $\CFKF(K)$ for the knot Floer complex over the ring $\F[U,V]$, following the notation of \cite[Section 2]{DHST}. (This invariant contains the same information as $\CFKi(K)$, as described in \cite[Section 1.5]{Zemke:2019-1}.) The invariant $\CFKF(K)$ is a finitely generated $\F[U,V]$-module. Recall that we say a knot $K$ has a \emph{unit box} if $\CFKF(K)$ contains a direct summand of the form $\{a, b, c, d\}$ with $$\d a= Ub + V c, \d b = Vd, \text{ and } \d c = U d.$$
The minus version of knot Floer homology, denoted by $\HFKm(K)$, is defined as the homology of the chain complex obtained from $\CFKF(K)$ by setting $V=0$. Note that $\HFKm(K)$ is a finitely generated $\F[U]$-module. Given an  $\F[U]$-module $M$, we define
$$\Ord_U(M):= \min \, \{k \in N \mid U^k \cdot \tors(M)=0\} \in \mathbb{N} \cup \{\infty\}.$$
The following knot invariant was defined in \cite{JMZ}.

\begin{definition}[{\cite[Definition 1.1]{JMZ}}]\label{def:order} If $K$ is a knot in $S^3$, we define its \emph{torsion order} as $$\Ord_U(K):= \Ord_U(\HFKm(K)).$$
\end{definition}

Moreover, recall that the torsion order of a ribbon knot gives a lower bound on the fusion number \cite[Corollary 1.8]{JMZ}. The goal of this section is to understand the behavior of torsion order under cabling. To be more precise, we will prove Proposition~\ref{prop:introcables}, which we restate here.

\begin{customprop}{~\ref{prop:introcables}}
If $K$ has a unit box, then 
$\Ord_U(K_{p_1, q_1; p_2, q_2; \dots ; p_n, q_n}) \geq p_1  p_2 \dots  p_n$.
\end{customprop}

\noindent We have the following immediate corollary.

\begin{corollary}\label{cor:fusion}
If $K$ is ribbon and has a unit box, then $\f(K_{p_1, 1; \dots ; p_n, 1}) \geq p_1 p_2 \dots p_n$.
\end{corollary}

\begin{proof}
By \cite[Corollary 1.8]{JMZ}, we have $\f(J) \geq \Ord_U(J)$ for any ribbon knot $J$.
\end{proof}

The proof of Proposition~\ref{prop:introcables} relies on bordered Floer homology \cite{LOT}, interpreted as immersed curves as in \cite{HRW}. If $K$ is a knot in $S^3$, denote the complement of an open tubular neighborhood of $K$ in $S^3$ by $M_K$.  Given a 3-manifold $M$ with torus boundary, the immersed multicurve invariant, denoted by $\ic(M)$, was defined in \cite[Theorem 1]{HRW}. Further, it was shown that one can recover $\HFKm(K)$ from $\ic(M_K)$ \cite[Theorem 51]{HRW:2018}. Hence if $\ic(M_K)$ is given, then we can compute invariants that come from $\HFKm(K)$.

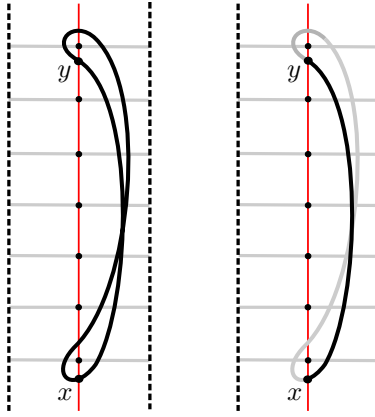
\begin{figure}[ht]
\includegraphics[scale=1.5]{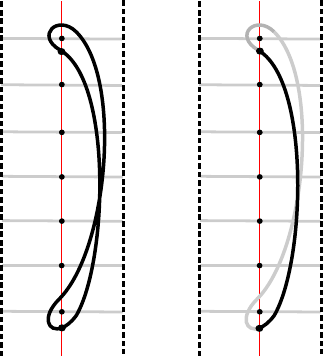}
\put(-35,4){\small$x$}
\put(-35,126){\small$y$}
\put(-121,4){\small$x$}
\put(-121,126){\small$y$}

\caption{Left, the immersed curve  intersecting $p=7$ vertical gray lines.  Right, the segment of the immersed curve that contributes an $\F[U]/U^p$ summand to $\HFKm$.}
\label{fig:figureprop}
\end{figure}

\begin{lemma}\label{lem:ordercable} 
Let $K$ be a knot in $S^3$. If $\ic(M_K)$ contains a curve as in Figure~\ref{fig:figureprop}, then $\Ord_U(K) \geq p$.
\end{lemma}
\begin{proof}
In order to translate from the immersed multicurve $\ic(M_K)$ to $\HFKm(K)$, we apply Theorem 51 of \cite{HRW:2018}. (See \cite[Figure 36]{HRW:2018} for an example illustrating Theorem 51.) A segment of an immersed curve lying entirely to the right of $p$ pegs (see the right of Figure~\ref{fig:figureprop}) tells us that $\CFKF(K)/(V=0)$ has a summand generated by $x$ and $y$ with
$$\d^- x= U^py.$$
Here, $p$ is just the number of times that the segment intersects the vertical gray lines. This implies that $y$ generates a $\F[U]/U^p$ summand in $\HFKm(K)$ and we conclude $\Ord_U(K) \geq p$.
\end{proof}

We combine \cite[Theorem 1]{HW-cables}, which provides an explicit algorithm for computing $\ic(M_{K_{p,q}})$ from $\ic(M_K)$, with Lemma~\ref{lem:ordercable} to prove Proposition~\ref{prop:introcables}. Before proceeding to the proof of the proposition, we give some examples of knots with unit boxes.

\begin{example}
A direct computation from a genus one doubly pointed Heegaard diagram shows that the figure eight knot has a unit box. See Figure \ref{fig:4_1}. More generally, it follows from \cite[Lemma 7]{Petkova} that any non-trivial thin knot $K$ with $\tau(K)=0$ has a unit box.
\end{example}

\begin{figure}[htb!]
\subfigure[]{
\begin{tikzpicture}[scale=0.7]

	\begin{scope}[thin, black!40!white]
		\draw [<->] (-2, 0) -- (2, 0);
		\draw [<->] (0, -2) -- (0, 2);
	\end{scope}
	\draw[step=1, black!50!white, very thin] (-1.9, -1.9) grid (1.9, 1.9);
	\node[] at (2.2, 0) (){\small{$A$}};	
	\node[] at (0, 2.3) (){\small{$m$}};	

	\node[] at (-1, -1) (a){$\F$};
	\node[] at (0, 0) (b){$\F^3$};
	\node[] at (1, 1) (c){$\F$};
	
	\draw [->] (b) -- (a);
	\draw [->] (c) -- (b);

\end{tikzpicture}
\label{subfig:}
}
\subfigure[]{
\begin{tikzpicture}[scale=0.7]

	\begin{scope}[thin, black!40!white]
		\draw [<->] (-2, 0) -- (2, 0);
		\draw [<->] (0, -2) -- (0, 2);
	\end{scope}
	\draw[step=1, black!50!white, very thin] (-1.9, -1.9) grid (1.9, 1.9);
	\node[] at (2.2, 0) (){\small{$i$}};	
	\node[] at (0, 2.3) (){\small{$j$}};	

	\filldraw (-1, -1) circle (2pt) node[] (){};
	\filldraw (0, 0) circle (2pt) node[] (){};
	\filldraw (1, 1) circle (2pt) node[] (){};
	
	\filldraw (-0.85, -0.85) circle (2pt) node[] (a){};
	\filldraw (-0.85, -0.15) circle (2pt) node[] (b){};
	\filldraw (-0.15, -0.85) circle (2pt) node[] (c){};
	\filldraw (-0.15, -0.15) circle (2pt) node[] (d){};
	\draw [->] (b) -- (a);
	\draw [->] (c) -- (a);
	\draw [->] (d) -- (b);
	\draw [->] (d) -- (c);
	
	\filldraw (0.15, 0.15) circle (2pt) node[] (a){};
	\filldraw (0.15, 0.85) circle (2pt) node[] (b){};
	\filldraw (0.85, 0.15) circle (2pt) node[] (c){};
	\filldraw (0.85, 0.85) circle (2pt) node[] (d){};
	\draw [->] (b) -- (a);
	\draw [->] (c) -- (a);
	\draw [->] (d) -- (b);
	\draw [->] (d) -- (c);

	\filldraw (1.4, 1.4) circle (0.7pt) node[] (){};	
	\filldraw (1.5, 1.5) circle (0.7pt) node[] (){};	
	\filldraw (1.6, 1.6) circle (0.7pt) node[] (){};	

	\filldraw (-1.4, -1.4) circle (0.7pt) node[] (){};	
	\filldraw (-1.5, -1.5) circle (0.7pt) node[] (){};	
	\filldraw (-1.6, -1.6) circle (0.7pt) node[] (){};	

\end{tikzpicture}
\label{subfig:}
}
\subfigure[]{
\includegraphics[scale=2]{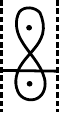}
}
\caption{Left, $\HFKh(4_1)$ in the Alexander-Maslov plane. Center, $\CFKi(4_1)$ in the $(i,j)$-plane. Right, the immersed curves associated to the complement of $4_1$.}
\label{fig:4_1}
\end{figure}
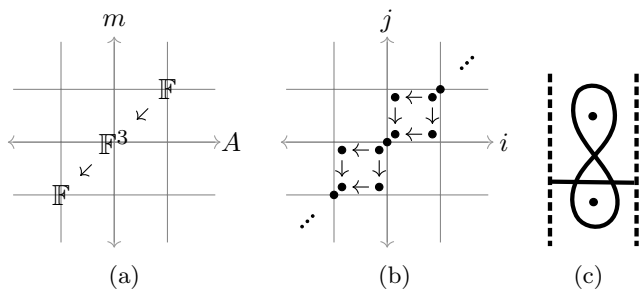

\begin{example}\label{ex:KT}
The Kinoshita-Terasaka knot $11n42$ and the Conway knot $11n34$ both have a unit box. By \cite{BaldwinGillam}, we have that $\HFKh(11n42)$ and $\HFKh(11n34)$ each have rank $33$ with higher differentials as in Figure~\ref{fig:KTConway}. Then a computation similar to the one in the proof of \cite[Lemma 7]{Petkova} yields that $\CFKF(11n42)$ and $\CFKF(11n34)$ each consist of a direct sum of a single generator and eight unit boxes. See Example 49 of \cite{HRW:2018} for the immersed curves associated to the Kinoshita-Terasaka and Conway knots.
\end{example}

\begin{figure}[htb!]
\subfigure[]{
\begin{tikzpicture}[scale=0.7]

	\begin{scope}[thin, black!40!white]
		\draw [<->] (-3, 0) -- (3, 0);
		\draw [<->] (0, -4) -- (0, 3);
	\end{scope}
	\draw[step=1, black!50!white, very thin] (-2.9, -3.9) grid (2.9, 2.9);
	\node[] at (3.2, 0) (){\small{$A$}};	
	\node[] at (0, 3.3) (){\small{$m$}};

	\node[] at (-2, -3) (a){$\F$};
	\node[] at (-2, -2) (b){$\F$};
	\node[] at (-1, -2) (c){$\F^4$};
	\node[] at (-1, -1) (d){$\F^4$};
	\node[] at (0, -1) (e){$\F^6$};
	\node[] at (0, 0) (f){$\F^7$};
	\node[] at (1, 0) (g){$\F^4$};
	\node[] at (1, 1) (h){$\F^4$};
	\node[] at (2, 1) (i){$\F$};
	\node[] at (2, 2) (j){$\F$};

	\draw [->] (c) -- (a);
	\draw [->] (e) -- (c);
	\draw[transform canvas={yshift=0.5ex, xshift=-0.5ex}, ->] (e) -- (c);
	\draw[transform canvas={yshift=-0.5ex, xshift=0.5ex}, ->] (e) -- (c);
	\draw [->] (g) -- (e);
	\draw[transform canvas={yshift=0.5ex, xshift=-0.5ex}, ->] (g) -- (e);
	\draw[transform canvas={yshift=-0.5ex, xshift=0.5ex}, ->] (g) -- (e);
	\draw [->] (i) -- (g);
	\draw [->] (d) -- (b);
	\draw [->] (f) -- (d);
	\draw[transform canvas={yshift=0.5ex, xshift=-0.5ex}, ->] (f) -- (d);
	\draw[transform canvas={yshift=-0.5ex, xshift=0.5ex}, ->] (f) -- (d);
	\draw [->] (h) -- (f);
	\draw[transform canvas={yshift=0.5ex, xshift=-0.5ex}, ->] (h) -- (f);
	\draw[transform canvas={yshift=-0.5ex, xshift=0.5ex}, ->] (h) -- (f);
	\draw [->] (j) -- (h);

\end{tikzpicture}
\label{subfig:KT}
}
\subfigure[]{
\begin{tikzpicture}[scale=0.7]

	\begin{scope}[thin, black!40!white]
		\draw [<->] (-4, 0) -- (4, 0);
		\draw [<->] (0, -5) -- (0, 4);
	\end{scope}
	\draw[step=1, black!50!white, very thin] (-3.9, -4.9) grid (3.9, 3.9);
	\node[] at (4.2, 0) (){\small{$A$}};	
	\node[] at (0, 4.3) (){\small{$m$}};

	\node[] at (-3, -4) (a){$\F$};
	\node[] at (-3, -3) (b){$\F$};
	\node[] at (-2, -3) (c){$\F^3$};
	\node[] at (-2, -2) (d){$\F^3$};
	\node[] at (-1, -2) (e){$\F^3$};
	\node[] at (-1, -1) (f){$\F^3$};
	\node[] at (0, -1) (g){$\F^2$};
	\node[] at (0, 0) (h){$\F^3$};
	\node[] at (1, 0) (i){$\F^3$};
	\node[] at (1, 1) (j){$\F^3$};
	\node[] at (2, 1) (k){$\F^3$};
	\node[] at (2, 2) (l){$\F^3$};
	\node[] at (3, 2) (m){$\F$};
	\node[] at (3, 3) (n){$\F$};

	\draw [->] (c) -- (a);
	\draw[transform canvas={yshift=0.25ex, xshift=-0.25ex}, ->] (e) -- (c);
	\draw[transform canvas={yshift=-0.25ex, xshift=0.25ex}, ->] (e) -- (c);
	\draw [->] (g) -- (e);
	\draw [->] (i) -- (g);
	\draw[transform canvas={yshift=0.25ex, xshift=-0.25ex}, ->] (k) -- (i);
	\draw[transform canvas={yshift=-0.25ex, xshift=0.25ex}, ->] (k) -- (i);
	\draw [->] (m) -- (k);
	\draw [->] (d) -- (b);
	\draw[transform canvas={yshift=0.25ex, xshift=-0.25ex}, ->] (f) -- (d);
	\draw[transform canvas={yshift=-0.25ex, xshift=0.25ex}, ->] (f) -- (d);
	\draw [->] (h) -- (f);
	\draw [->] (j) -- (h);
	\draw[transform canvas={yshift=0.25ex, xshift=-0.25ex}, ->] (l) -- (j);
	\draw[transform canvas={yshift=-0.25ex, xshift=0.25ex}, ->] (l) -- (j);
	\draw [->] (n) -- (l);

\end{tikzpicture}
\label{subfig:Conway}
}
\caption{Left, $\HFKh(11n42)$ in the Alexander-Maslov plane. Right,  $\HFKh(11n34)$. The arrows represent the higher differentials.}
\label{fig:KTConway}
\end{figure}
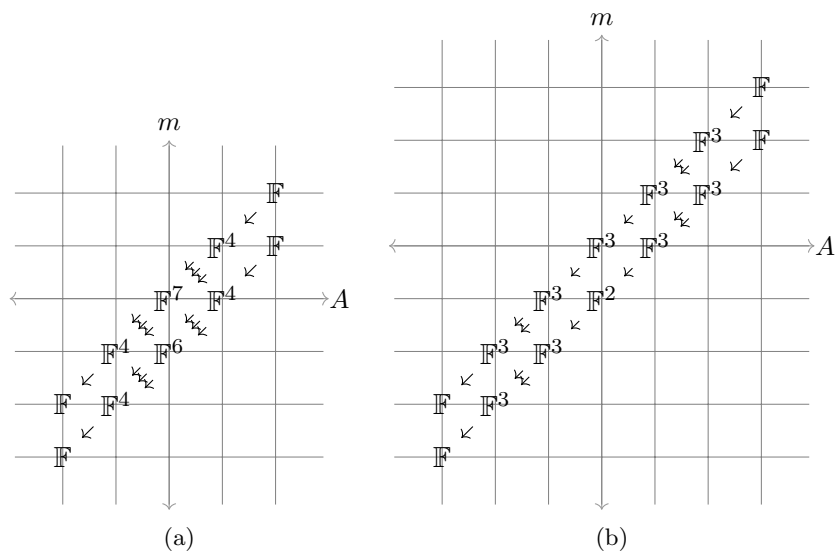

\begin{proof}[Proof of Proposition \ref{prop:introcables}]
We first prove $\Ord_U(K_{p,q}) \geq p$. By \cite[Proposition 47]{HRW:2018} (see in particular the third diagram in Figure~10 of \cite{HRW:2018}), we have that if $K$ has a unit box, then the immersed multicurve for $M_K$ contains a figure eight curve, as in the leftmost diagram in Figure~\ref{fig:figureeight}. 

We now apply the three step process following Theorem 1 of \cite{HW-cables}. (Since our curve does not have loose ends, we may skip step (2).) Namely, we first draw $p$ copies of the figure eight curve next to each other, each scaled vertically by a factor of $p$, staggered in height such that each copy of the curve is a height of $q$ units lower than the previous copy. We then translate the pegs horizontally so that they lie in the same vertical line, carrying the curves along with them. See Figure~\ref{fig:figureeight}.

\begin{figure}
\includegraphics[scale=1.5]{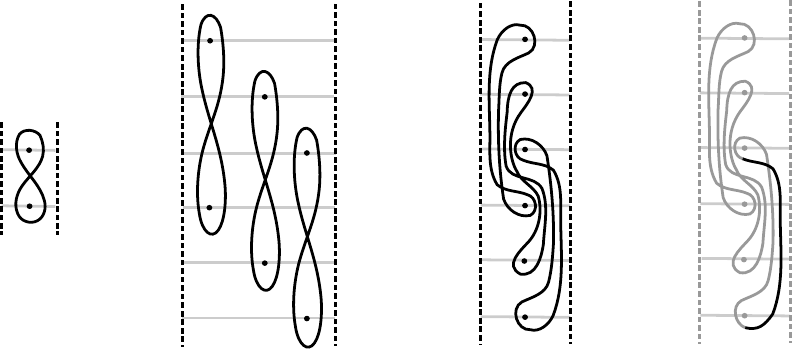}
\caption{The $(3,1)$-cable of a unit box. Far left, the immersed curve associated to a unit box. Second from left, $p=3$ copies of the immersed curve, each scaled by a factor of $p$, staggered in height such that each copy of the curve is $q=1$ units lower than the previous copy. Second from right, the result after the pegs are translated horizontally so that they lie in the same vertical line. Right, the segment of the immersed curve that contributes an $\F[U]/U^p$ summand to $\HFKm$.}
\label{fig:figureeight}
\end{figure}

Note that the rightmost copy of the $p$-scaled copy of the figure eight curve is the same as the immersed curve in Figure~\ref{fig:figureprop}. Hence by Lemma~\ref{lem:ordercable}, we conclude that $\Ord_U(K_{p,q}) \geq p$, as desired.

For iterated cables, we simply repeat this argument. That is, we consider the rightmost $p_1$-scaled copy of the figure eight curve, and take $p_2$ copies of it, each scaled vertically by a copy of $p_2$ and staggered in height such that each copy is $q_2$ units lower than the previous copy. An identical argument to the one above shows that the rightmost copy of the $p_2$-scaled curve will have a segment that lies entirely to the right of $p_1 p_2$ pegs. Iterating this argument shows that $\Ord_U(K_{p_1,q_1; \dots ; p_n, q_n}) \geq p_1 p_2 \dots p_n$, as desired.\end{proof}

Before we prove the main theorem, we show that the knot Floer complex over $\F[U,V]$ of any slice knot $K$ with $\Ord_U(K)=1$ splits as a direct sum of $\F[U,V]$ and unit boxes. Recall that since the torsion order gives lower bound for the fusion number, if a ribbon knot has fusion number one, then it has torsion order one (recall that knot Floer homology detects the unknot \cite{OSgenus}). Hence the following proposition in particular shows that the knot Floer complex over $\F[U,V]$ of any ribbon knot with fusion number one splits as a direct sum of $\F[U,V]$ and unit boxes.

\begin{proposition}\label{prop:ord1box}
If $K$ is a slice knot with $\Ord_U(K)=1$, then $K$ has a unit box. Moreover, $\CFKF(K)$ splits as a direct sum of $\F[U,V]$ and unit boxes.
\end{proposition}

\begin{proof}
Since $K$ is slice, by \cite[Theorem 1]{Hom:2017-1} (see also \cite[Theorem 1.5]{Zemke:invol}, forgetting the involutive part), we have that $\CFKF(K)$ splits as $\F[U,V] \oplus A$ for some free, finitely generated differential graded $\F[U,V]$-module $A$. 

At times, it will be convenient to consider $A/(U=0)$, which is naturally an $\F[V]$-module. We denote the induced differential on $A/(U=0)$ by $\d_V$. Similarly, denote the induced differential on $A/(V=0)$ by $\d_U$. The module $H_*(A/(U=0))$ is annihilated by $U^{\Ord_U(K)}$ and the module $H_*(A/(V=0))$ is annihilated by $V^{\Ord_U(K)}$. (Recall that by symmetry between $U$ and $V$, $\Ord_U$ can be defined in terms of $H_*(\CFKF(K)/(U=0))$ or equivalently $H_*(\CFKF(K)/(V=0))$.)

Suppose that $A$ is generated over $\F[U,V]$ by $\{x_i\}_{i=1}^n$. We may assume that $A$ is reduced, i.e., $\d x_i$ is in the image of $U$ or $V$ (or possibly both) for all $i$. Note that the notion of reducedness is preserved under basis changes of the form $x_i \mapsto x_i + \sum_{j\in J} U^{\ell_j} V^{m_j} x_{i_j}$. We will use the grading conventions of \cite{Zemke:2019-1}; see also \cite[Section 2]{DHST}. Recall that multiplication by $U$ lowers the Alexander grading by $1$, multiplication by $V$ raises the Alexander grading by $1$, and the differential preserves the Alexander grading.

Without loss of generality, let $x_1$ be a generator in Alexander grading $g(K)$. Since knot Floer homology detects genus \cite{OSgenus} and $A$ is reduced, there are no $x_i$ in Alexander grading greater than $g(K)$. Hence no $V$-power of $x_1$ (that is, no $V^nx_1$ for any natural number $n$) is in the image of $\d_V$. Since $H_*(A/(U=0))$ is $V$-torsion, it follows that $\d_V x_1 \neq 0$. Since $\Ord_U(K)=1$, we have that $\d_V x_1 = Vx_{i_0} + V \sum_{j \in J} V^{n_j}x_{i_j}$ for some index set $J$, where $n_j \geq 0$ for each $j$. We may reorder our basis elements such that $i_0 = 2$ and we may perform a change of basis to replace $x_2$ with $x_2 + \sum_{j \in J} V^{n_j} x_{i_j}$, so that 
\[ \d_V x_1 = Vx_2. \]

We now consider how $\d_U$ interacts with $x_1$. Since the Alexander grading of $x_1$ is $g(K)$, it follows that $\d_U x_1 = 0$. Because $H_*(A/(V=0))$ is $U$-torsion, we have that $U^nx_1 \in \Im \d_U$ for some natural number $n$. Moreover, since $\Ord_U(K) = 1$, we have that $n=1$. Hence $\d_U$ of some linear combination of the form $x_i + \sum_{j\in J} U^{\ell_j} x_{i_j}$ is equal to $U x_1$, and so (after a possible basis change) we may assume that there is a basis element  $x_3$ such that
\[ \d_U x_3 = U x_1. \] 

We now consider $\d^2 x_3$. We have that $\d_U x_3 = U x_1$ and $\d_V x_1 = V x_2$. Recall that $\d_U$ is the induced differential on $A/(V=0)$; hence 
\[ \d x_3 = U x_1 + V \sum_{j \in J_1} U^{\ell_j} V^{m_j} x_{i_j} \]
for some index set $J_1$, where  $\ell_j, m_j \geq 0$ for each $j$. Similarly, 
\[ \d x_1 = V x_2 + U \sum_{j \in J_2} U^{p_j} V^{q_j} x_{i_j} \]
for some index set $J_2$, where $p_j, q_j \geq 0$ for each $j$. (In fact, since $A(x_1) = g(K)$, we have that $\ell_j \leq m_j + 2$ and $p_j +1 \leq q_j$.) Thus, 
\[ \d^2 x_3 = UV x_2 + U^2 \sum_{j \in J_2} U^{p_j} V^{q_j} x_{i_j} + V \sum_{j \in J_1} U^{\ell_j} V^{m_j} \d x_{i_j}. \]
Since $\d^2 x_3 = 0$, we must have $q_j \ge 1$ for all $j \in J_2$, and also $\ell_{j'} \leq 1$ and $m_{j'} = 0$ for some $j' \in J_1$. Moreover, since our basis for $A$ is reduced, we in fact have that $\ell_{j'} = m_{j'} = 0$ for some $j' \in J_1$. After possibly reordering our basis elements, we may assume that $i_{j'} = 4$. We may then perform a change of basis and replace $x_4$ with $\sum_{j \in J_1} U^{\ell_j} V^{m_j} x_{i_j}$, so that $\d x_3 = U x_1 + V x_4$. We may also replace $x_2$ with $x_2 + U \sum_{j \in J_2} U^{p_j} V^{q_j-1} x_{i_j}$. Then setting 
\[ a = x_3, \quad b = x_1, \quad c = x_4 \quad d = x_2, \] 
we see that 
\[ \d a = U b + V c, \quad \d b = V d, \quad \d c = U d. \] 
It is straightforward to check that $\{a, b, c, d\}$ generate a direct summand of $A$. Hence $K$ has a unit box.

\begin{figure}[htb!]
\begin{tikzpicture}[scale=1]

	\begin{scope}[thin, black!40!white]
		\draw [<->] (-2, -1) -- (2, -1);
		\draw [<->] (0, -2) -- (0, 2);
	\end{scope}
	\draw[step=1, black!50!white, very thin] (-1.9, -1.9) grid (1.9, 1.9);
	\node[] at (2.2, -1) (){\small{$i$}};	
	\node[] at (0, 2.3) (){\small{$j$}};	
	
	\filldraw (0.1, 0.1) circle (2pt) node[label = left:{$x_2$}] (a){};
	\filldraw (0.1, 0.9) circle (2pt) node[label = left:{$x_1$}] (b){};
	\filldraw (0.9, 0.1) circle (2pt) node[label = right:{$x_4$}] (c){};
	\filldraw (0.9, 0.9) circle (2pt) node[label = right:{$x_3$}] (d){};
	\draw [->] (b) -- (a);
	\draw [->] (c) -- (a);
	\draw [->] (d) -- (b);
	\draw [->] (d) -- (c);

\end{tikzpicture}
\caption{The unit box from the proof of Proposition \ref{prop:ord1box}, drawn as a direct summand of $\CFKi(K)$ in the $(i, j)$-plane. Here, $g(K) = 2$.}
\label{fig:}
\end{figure}
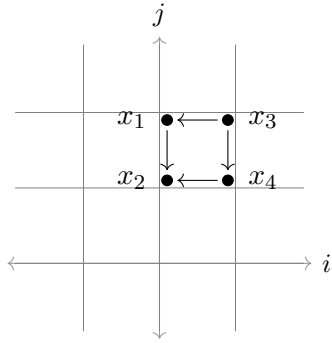

We now inductively apply the above change of basis to show that $A$ splits as a direct sum of unit boxes. We have that $A$ is generated over $\F[U,V]$ by $\{x_i\}_{i=1}^n$ and we have shown that $\{x_1, x_2, x_3, x_4\}$ generate a direct summand of $A$. Without loss of generality, let $x_5$ be a generator of maximal Alexander grading in $\{x_i\}_{i=5}^n$. We can now apply the above argument verbatim, with $x_5$ playing the role of $x_1$, to split off a unit box generated by $x_5, x_6, x_7, x_8$. Repeated applications of this argument shows that $A$ splits a direct sum of unit boxes, as desired.
\end{proof}

Finally, we prove Theorem~\ref{thm:main}.

\begin{proof}[Proof of Theorem 1.1]
Suppose $K$ is a ribbon knot with fusion number one. As mentioned above, since the torsion order gives lower bounds for the fusion number and the knot Floer homology detects the unknot, we have $\Ord_U(K)=1$. By Proposition~\ref{prop:ord1box}, we see that $K$ has a unit box. Moreover, Proposition~\ref{prop:handle} and Corollary~\ref{cor:fusion} imply that $\mathcal{F}(K_{p,1})=p$.

For the strong homotopy fusion number, Proposition~\ref{prop:handle} implies that $\mathcal{F}_{sh}(K_{p,1})\le 1$. Note that a knot has strong homotopy fusion number zero if and only if it is the unknot \cite{Gabai:1987-1}. Hence $\mathcal{F}_{sh}(K_{p,1})=1$, as desired.

The statement for iterated cables follows analogously.
\end{proof}

\bibliographystyle{alpha}
\def\MR#1{}
\bibliography{bib}
\end{document}